\documentclass[12pt]{article}
\usepackage[cp1251]{inputenc}
\usepackage{graphicx}
\usepackage[english,russian]{babel}
\usepackage{amsmath,amssymb,euscript, amsthm,amsfonts,mathrsfs,amscd}
\theoremstyle{plain}
\newtheorem{thm}{Theorem}
\newtheorem{prop}{Proposition}
\usepackage{color}

\theoremstyle{definition}
\newtheorem{deff}{Definition}
\newtheorem{nota}{Notation}

\theoremstyle{remark}
\newtheorem{rem}{Remark}

\newcommand{\bb}{\color{blue}}
\newcommand{\rr}{\color{red}}
\newcommand{\bk}{\color{black}}

\newcommand{\dd}{\ensuremath{\displaystyle}}

\newcommand{\supl}{\ensuremath{\sup\limits}}
\newcommand\intl{\ensuremath{\int\limits}}

\newcommand{\PP}{\ensuremath{\mathbf{P}}}
\newcommand{\PPP}{\ensuremath{\mathscr{P}}}

\newcommand{\UU}{\ensuremath{\mathscr{U}}}

\newcommand{\EE}{\mathsf E}
\newcommand{\EEE}{\mathcal E}

\newcommand{\1}{\ensuremath{\mathbf{1}}}

\newcommand{\ud}{\,\mathrm{d}}

\newcommand{\bd}{\stackrel{\text{\rm def}}{=\!\!\!=}}

\usepackage{mathrsfs}
\usepackage[ left=20mm, top=20mm, right=20mm, bottom=20mm, footskip=10mm, nohead, nomarginpar, driver=xetex]{geometry}

\title{
\bf On strong bounds of rate of convergence for regenerative processes}
\author{G.~A.~Zverkina \thanks{The author is supported by the RFBR, project No No 17-01-00633 A.}}

\begin{document}
\selectlanguage{english}
\maketitle
\begin{abstract}
We give strong bounds for the rate of convergence of the regenerative process distribution to the stationary distribution in the total variation metric.
These bounds are obtained by using coupling method.
We propose this method for obtaining such bounds for the queueing regenerative processes.
\end{abstract}

\section{Introduction}
We study the rate of convergence of distribution of regenerative process to the stationary distribution in the total variation metric.

Many queueing processes are regenerative, and establishing bounds for the rate of their convergence is a very important problem for the practical applications of the queueing theory. Recall the definition of regenerative process. \\
 \begin{deff} \label{zvlab1} 
The process $\left(X_t, \,t \geqslant 0\right)$ on a probability space $\left(\Omega, \mathscr{F} , \mathbf{P}\right)$, with a measurable state space $\left(\mathscr{X} , \mathscr{B} \left(\mathscr{X}\right)\right)$ is regenerative, if there exists an increasing sequence $ \left \{ \theta_n \right \} $ $\left(n \in \mathbb{Z} _+\right)$ of Markov moments with respect to the filtration $ \mathscr{F} _{t \geqslant 0} $ such that the sequence
$$
 \left \{ \Theta_n \right \} = \left \{ X_{t+ \theta_{n-1} } -X_{ \theta_{n-1} } , \theta_n- \theta_{n-1} ,t \in[ \theta_{n-1} , \theta_n) \right \} , \quad n \in \mathbb{N}
$$
consists of independent identically distributed (i.i.d.) random elements on $\left(\Omega, \mathscr{F} , \mathbf{P}\right)$.
If $ \theta_0 \neq 0$, then the process $\left(X_t, \,t \geqslant 0\right)$ is called delayed.
 \end{deff}

Denote $ \zeta_n \stackrel { \rm{ \;def} } {= \! \! \!=} \theta_n- \theta_{n-1} $, and let $F\left(s\right)= \mathbf{P} \left \{ \zeta_n \leqslant s \right \} = \mathbf{P} \left \{ \zeta_1 \leqslant s \right \} $ ($n \in \mathbb{N} $) be the distribution function of regeneration period; we suppose that the distribution $F$ is not lattice.

Also denote $ \zeta_0 \stackrel{ \rm{ \;def} } {= \! \! \!=} \theta_0$, $ \mathbb{F} \left(s\right) \stackrel{ \rm{ \;def} } {= \! \! \!=} \mathbf{P} \left \{ \zeta_0 \leqslant s \right \} $.

Denote $ \mathscr{P} _t^{X_0} \left(A\right)= \mathbf{P} \{X_t \in A \} $ for the process $\left(X_t,\,t\geqslant 0\right)$ with the initial state $X_0$.

If $ \mathbf{E} \, \zeta_i< \infty$, then for all $X_0$ we have $ \mathscr{P} _t^{X_0} \Longrightarrow \mathscr{P} $ where $ \mathscr{P} $ is the stationary distribution of the process $\left(X_t,\,t\geqslant 0\right)$.

In the 70s A.A.~Borovkov showed that if the condition
$$
 \mathbf{E}\,e^{\dd\alpha \zeta}<\infty,\;\; \alpha>0
$$
is satisfied, then for all set $A\in\mathscr{B}\left(\mathscr{X}\right)$ and for all $a<\alpha$ there exists $\mathfrak{R}\left(a,A\right)$ such that for all $t>0$:
$$
|\mathbf{P}\{X_t\in A\}-\mathscr{P}\left(A\right) |<\frac{\mathfrak{R}\left(a,A\right)}{e^{\dd a t}},
$$
and if the condition
$$
 \qquad\mathbf{E}\,\zeta^K<\infty,\;\; k>1
$$
is satisfied, then for all set $A\in\mathscr{B}\left(\mathscr{X}\right)$ and for all $k<K-1$ there exists $R\left(k,A\right)$ such that for all $t>0$:
$$
 |\mathbf{P}\{X_t\in A\}-\mathscr{P}\left(A\right) |<\frac{R\left(k,A\right)}{t^{\kappa}},
$$
where $\mathscr{P}$ is the stationary distribution of $\left(X_t,\,t\geqslant 0\right)$, but bounds of the constants $\mathfrak{R}\left(\cdot\right)$ and $R\left(\cdot\right)$ are unknown.

Later this results was repeated and extended by probabilistic approaches, namely, by \emph{modified} coupling method: shift-coupling and distributional shift-coupling \cite{Thor,ThorDistr}, $\varepsilon$-coupling \cite{Kalash,kalashreg}.

So, now it is known that if $ \mathbf{E} \, \zeta_n^K< \infty$ for some $K>1$ then for every $k < K-1$ and $X_0 \in \mathscr{X} $
$$
 \lim \limits_{t \to \infty} t^k \left \| \mathscr{P} _t^{X_0} - \mathscr{P} \right \|_{TV} =0,
$$
and if $ \mathbf{E}\, e^{\dd\alpha \zeta_i}< \infty$ for some $ \alpha>0$ then for every $a< \alpha$ and $X_0 \in \mathscr{X} $
$$
\lim \limits_{t \to \infty} e^{\dd at} \left \| \mathscr{P} _t^{X_0} - \mathscr{P} \right \|_{TV} =0
$$
(see, e.g., \cite{Down,Kalash,kalashreg,Lindvall,Silv,Thor,ThorDistr} et al.).

So, we know two propositions:
\begin{prop}
 If $ \mathbf{E} \, \zeta_n^K< \infty$ for some $K>1 $, then for some $k < K-1$ and for all $X_0 \in \mathscr{X} $ there exists $C\left(X_0,k\right)$ such that
$$ \left \| \mathscr{P} _t^{X_0} - \mathscr{P} \right \|_{TV} \leqslant \frac{ C\left(X_0,k\right)}{\left(1+t\right)^{k}}.
$$
\end{prop}

\begin{prop}
 If $ \mathbf{E} \, e^{\dd \alpha \zeta_i}< \infty$ for some $ \alpha>1$, then for all $a< \alpha$ and $X_0 \in \mathscr{X} $ there exists $ \mathfrak{C} \left(X_0, \alpha\right)$ such that
 $$
 \left \| \mathscr{P} _t^{X_0} - \mathscr{P} \right \|_{TV} \leqslant \frac{\mathfrak{C} \left(X_0,a\right)}{e^{\dd at}}.
 $$
\end{prop}

Our goal is to find the bounds of the constants $C\left(X_0,k\right)$ and $ \mathfrak{C} \left(X_0,a\right)$ in sufficiently wide conditions; note that the behavior of the constants $C\left(X_0,k\right)$ and $ \mathfrak{C} \left(X_0,a\right)$ has been studied in \cite{Lund,Rob,VerAiT,VerSeva,vzMPRF,v-z2015} for some special cases of regenerative processes.

{ \it In the sequel, we suppose that }
 \begin{flushright}
 \fbox{$ \dd\int \limits_{ \{s: \, \exists F'\left(s\right) \} }^{\;} F'\left(s\right) \, \mathrm{d} \, s>0; \quad \mathbf{E} \, \zeta_i\bd \mu_i< \infty$} \hspace{4.8cm} ($ \ast$)
 \end{flushright}
 \section{Notations and the main results}

\begin{nota}
The times $ \theta_i$ (Definition \ref{zvlab1}) form the renewal process $N_t \stackrel{ \rm{ \;def} } {= \! \! \!=} \sum \limits_{i=0} ^ \infty \mathbf{1} \left(\theta_i \leqslant t\right)$.
Denote $B_t \stackrel{ \rm{ \;def} } {= \! \! \!=} t- \theta_{N_t-1} $ (with $ \theta_{-1} \stackrel{ \rm{ \;def} } {= \! \! \!=} 0$).
The process $\left(B_t,\, t\geqslant0\right)$ is the backward renewal time of the process $N_t$. \hfill\ensuremath{\triangleright}
\end{nota}
 \begin{rem}
The process $\left(B_t,\, t\geqslant0\right)$ is the Markov regenerative process. \hfill\ensuremath{\triangleright}
 \end{rem}
\begin{nota}
For nondecreasing function $F\left(s\right)$ we put $F^{-1} \left(y\right) \stackrel{ \rm{ \;def} } {= \! \! \!=} \inf \{x: \,F\left(x\right) \geqslant y \} $. \hfill\ensuremath{\triangleright}
\end{nota}

\begin{nota}
Denote
$$ \widetilde{F} \left(s\right) \stackrel{ \rm{ \;def} } {= \! \! \!=} \frac{\dd\int \limits_{0} ^s \left(1-F\left(u\right)\right) \, \mathrm{d} \, u}{\mu},
$$
where
$$ \mu \stackrel{ \rm{ \;def} } {= \! \! \!=} \int \limits_{0} ^ \infty u \, \mathrm{d} \, F\left(u\right)= \mathbf{E} \, \zeta.
$$\hfill\ensuremath{\triangleright}
\end{nota}

\begin{nota}
Denote
$F_a\left(s\right) \stackrel {{ \rm \; def} } {= \! \! \!=} \dd\frac{F\left(s+a\right)-F\left(a\right)} {1-F\left(a\right)} $; $\;\; \mu_0 \stackrel{ \rm{ \;def} } {= \! \! \!=} \mathbf{E} \, \zeta_0$. \hfill\ensuremath{\triangleright}
\end{nota}

\begin{nota} Let $\UU, \UU', \UU'', \UU''', \mathscr{U} _i, \mathscr{U} _i', \mathscr{U} _i'', \mathscr{U} _i'''$ be independent uniformly distributed on $[0,1)$ random variables on some probability space $\left(\widetilde{ \Omega} , \widetilde{ \mathscr{F} } , \widetilde{ \mathbf{P} }\right)$. \hfill\ensuremath{\triangleright}
\end{nota}

\begin{nota} Denote
$$
\varphi\left(s\right) \stackrel {{ \rm \; def} } {= \! \! \!=} \mathbf{1} \Big(\exists \, F'\left(s\right) \Big) \times \left (F'\left(s\right) \wedge \widetilde F'\left(s\right) \right)
$$
and
$$ \Phi\left(s\right) \stackrel {{ \rm \; def} } {= \! \! \!=} \int \limits_0^s \varphi\left(u\right) \, \mathrm{d} \, u;
$$
the condition ($ \ast$) implies
$ \varkappa \stackrel {{ \; \rm def} } {= \! \! \!=}\dd \int \limits_0^ \infty \varphi\left(s\right) \, \mathrm{d} \, s= \Phi\left(+ \infty\right)>0$. \hfill\ensuremath{\triangleright}
\end{nota}

\begin{nota} Denote
$$ \Psi\left(s\right) \stackrel {{ \rm \;def} } {= \! \! \!=} F\left(s\right)- \Phi\left(s\right),\qquad \widetilde \Psi\left(s\right) \stackrel {{ \rm \; def} } {= \! \! \!=} \widetilde F\left(s\right)- \Phi\left(s\right);
$$
then $ \Psi\left(+ \infty\right)= \widetilde \Psi\left(+ \infty\right)=1- \varkappa$.

Denote $\mathfrak{P}_a\left(\zeta_1\right)\bd \dd\intl_0^\infty e^{\dd a s}\ud\Psi\left(s\right)$. \hfill\ensuremath{\triangleright}
\end{nota}

\begin{rem} \label{rem1}
Note that $ \varkappa^{-1} \Phi\left(s\right)$ is the distribution function, and if $ \varkappa<1$ then $\left(1- \varkappa\right)^{-1} \Psi\left(s\right)$ and $\left(1- \varkappa\right)^{-1} \widetilde{ \Psi} \left(s\right)$ are the distribution function.

If $ \varkappa=1$, then $ \Phi\left(s\right) \equiv F\left(s\right) \equiv \widetilde{F} \left(s\right)=1-e^{\dd - \lambda s}$, and we put $ \Psi\left(s\right) \equiv \widetilde{ \Psi} \left(s\right) \equiv 0$; in this case we assume $\left(1- \varkappa\right)^{-1} \Psi\left(s\right) \equiv \left(1- \varkappa\right)^{-1} \widetilde{ \Psi} \left(s\right) \equiv 0$, and $ \Psi^{-1} \left(u\right)= \widetilde{ \Psi} ^{-1} \left(u\right)=0$.
\hfill\ensuremath{\triangleright}
\end{rem}

\begin{nota} Put for independent random variables $ \mathscr{U} , \mathscr{U} ', \mathscr{U} ''$ uniformly distributed on $[0,1)$
$$
 \Xi\left(\mathscr{U} , \mathscr{U} ', \mathscr{U} ''\right) \stackrel {{ \rm \; def} } {= \! \! \!=} \mathbf{1} \left(\mathscr{U} < \varkappa\right) \Phi^{-1} \left(\varkappa \mathscr{U} '\right)+ \mathbf{1} \left(\mathscr{U} \geqslant \varkappa\right) \Psi^{-1} \left(\left(1- \varkappa\right) \mathscr{U} ''\right);
$$
$$
 \widetilde \Xi\left(\mathscr{U} , \mathscr{U} ', \mathscr{U} ''\right) \stackrel {{ \; \rm def} } {= \! \! \!=} \mathbf{1} \left(\mathscr{U} < \varkappa\right) \Phi^{-1} \left(\varkappa \mathscr{U} '\right)+ \mathbf{1} \left(\mathscr{U} \geqslant \varkappa\right) \widetilde \Psi^{-1} \left(\left(1- \varkappa\right) \mathscr{U} ''\right).
$$
\hfill\ensuremath{\triangleright}
\end{nota}

\begin{rem}
It is easy to see that
$$
F\left(s\right)= \varkappa \left (\varkappa^{-1} \Phi\left(s\right) \right)+ \left(1- \varkappa\right) \left (\left(1- \varkappa\right)^{-1} \Psi\left(s\right) \right)
$$
and
$$ \widetilde{F} \left(s\right)= \varkappa \left (\varkappa^{-1} \Phi\left(s\right) \right) +\left(1- \varkappa\right) \left (\left(1- \varkappa\right)^{-1} \widetilde{ \Psi} \left(s\right) \right).
$$
Hence,
$$
 \mathbf{P} \{ \Xi\left(\mathscr{U} , \mathscr{U} ', \mathscr{U} ''\right) \leqslant s \} =F\left(s\right),
 \qquad
 \mathbf{P} \{ \widetilde{ \Xi} \left(\mathscr{U} , \mathscr{U} ', \mathscr{U} ''\right) \leqslant s \} = \widetilde{F} \left(s\right),
$$
and
$$
 \mathbf{P} \{ \Xi\left(\mathscr{U} , \mathscr{U} ', \mathscr{U} ''\right)= \widetilde{ \Xi} \left(\mathscr{U} , \mathscr{U} ', \mathscr{U} ''\right) \} = \varkappa.
$$
 \hfill\ensuremath{\triangleright}
\end{rem}

\begin{nota}
Denote
\begin{multline*}
 C\left(X_0,k\right) \stackrel{ \rm{ \;def} } {= \! \! \!=} \mathbf{E} \, \zeta_0^k \varkappa \sum \limits_{n=1} ^ \infty \left (\left(n+1\right)^{k-1} \left(1- \varkappa\right)^{n-1} \right)+\\ \\
 +\mathbf{E} \, \zeta_1^k \sum \limits_{n=1} ^ \infty \left (\left (\varkappa n\left(n+2\right)^{k-1} + \left(n+1\right)^{k-1} \right) \left(1- \varkappa\right)^{n-1} \right)\Big[=\mathscr{C}\left(\zeta_0,k\right)\Big].
\end{multline*}
\hfill\ensuremath{\triangleright}
\end{nota}

\begin{nota} Denote $ \mathfrak{C} \left(X_0,a\right) \stackrel{ \rm{ \;def} } {= \! \! \!=} \dd\frac{ \mathfrak{P}_a\left(\zeta_1\right) \mathbf{E} \, e^{\dd a \zeta_0} } {1-\mathfrak{P}_a\left(\zeta_1\right)} \Big[=\mathcal{C}\left(\zeta_0,k\right)\Big]$. \hfill\ensuremath{\triangleright}
\end{nota}

 \begin{thm} \label{result} ~

\textbf{1. } If $ \mathbf{E} \,\left(\zeta_i\right)^K< \infty$  for some $K \geqslant 1$, then for all $k \in[1,K]$
$$
\left \| \mathscr{P} _t^{X_0} - \mathscr{P} \right \|_{TV} \leqslant 2\dd\frac{ C\left(X_0,k\right)}{t^{ \alpha}}.
$$

\textbf{2. } If $\EE\,e^{\dd \alpha\, \zeta_i}<\infty$, then there exists $a>0$ such that $\mathfrak{P}_a<1$, and
$$
\left \| \mathscr{P} _t^{X_0} - \mathscr{P} \right \|_{TV} \leqslant 2 \frac{\mathfrak{C} \left(X_0,a\right)}{e^{\dd a t}} .
$$
\end{thm}

\section{Idea of the proof of the Theorem \ref{result}.}

\subsection{Coupling method (see \cite{Lindvall-Lec}).} To prove the Theorem, we will use the \emph{modified coupling method.} The coupling method invented by W.~Doeblin in \cite{Doeb} is used to obtain the bounds of the rate of convergence of a Markov process to the stationary regime.

Let $\left(X_t',\,t\geqslant 0\right)$ and $\left(X_t'',\,t\geqslant 0\right)$ be two versions of Markov process $ \left(X_t,\,t\geqslant 0\right)$ with different initial states $x_0'$ and $X_0''$,
correspondingly, and put
$$ \mathscr{P} ^{X_0'} _t\left(A\right) \stackrel{ \rm{ \;def} } {= \! \! \!=} \mathbf{P} \{X_t' \in A \} ,\qquad \mathscr{P} ^{X_0''} _t\left(A\right) \stackrel{ \rm{ \;def} } {= \! \! \!=} \mathbf{P} \{X_t'' \in A \} ,
$$ and
$$
 \tau \left (X_0',X_0'' \right) \stackrel{ \rm{ \; def} } {= \! \! \!= \! \! \!=} \inf \left \{ t>0: \,X_t'=X_t'' \right \}.
$$

Let we know that $\PPP^{X_0}_t\Longrightarrow \PPP$ for every initial state $X_0$.

We suppose that $ \mathbf{E} \, \varphi \left (\tau \left (X_0',X_0'' \right) \right)=$ $C \left (X_0',X_0'' \right)< \infty$ for some positive increasing function $ \varphi\left(t\right)$.
Then
$$
 \left| \mathscr{P} _t^{X_0'} \left( A\right)- \mathscr{P} _t^{X_0''} \left( A\right) \right| \leqslant \mathbf P \left \{ \tau \left (X_0',X_0'' \right) >t \right \}
 =\mathbf P \left \{ \varphi \left (\tau \left (X_0',X_0'' \right) \right) > \varphi\left(t\right) \right \} \leqslant
 \frac{ \mathbf{E} \, \varphi \left (\tau \left (X_0',X_0'' \right) \right)}{\varphi\left(t\right)}
$$
by the coupling inequality and Markov inequality.

By integration of this inequality with respect to the stationary measure $ \mathscr{P} $ we have
\begin{equation} \label{zvlab2}
 \left| \mathscr{P} _t^{X_0'} \left( A\right)- \mathscr{P} \left(A\right) \right| \leqslant \left(\varphi\left(t\right)\right)^{-1} \int \limits_{ \mathscr{X} } \varphi \left (\tau \left (X_0',X_0'' \right) \right) \, \mathrm{d} \mathscr{P} \left (X_0'' \right)= \frac{\widehat C \left (X_0' \right)}{\varphi\left(t\right)},
\end{equation}
and
$$
 \left \| \mathscr{P} _t^{X_0'} - \mathscr{P} \right \|_{TV} \leqslant 2 \frac{\widehat{C} \left (X_0' \right)}{\varphi\left(t\right)}.
$$

Emphasize that the application of the coupling method is possible only for the Markov processes.
However, in queuing theory, usually the regenerative queueing processes are not Markov. Therefore, the state space of considered regenerative process must be extended so that the regenerative process with this state space would become Markov.

So, for the use of the coupling method for the arbitrary regenerative process $\left (X_t,\,t\geqslant 0\right)$ we must extend the state space $ \mathscr{X} $ of this process by such a way that the extended process $\left(\overline{X}_t,\,t\geqslant 0\right)$ with this extended state space $ \overline{ \mathscr{X} } $ is Markov.
For markovization of non-Markov regenerative process we can include to the state $X_t$, $t \in[ \theta_{n-1} , \theta_n)$ full history of the process on the time interval $[ \theta_{n-1} ,t]$: the extended process with the states $ \overline X_t \stackrel {{ \rm \; def} } {= \! \! \!=} \left \{ X_s, s \in[ \theta_{n-1} ,t] |t< \theta_n \right \} $ is Markov and regenerative with the extended state space $ \overline{ \mathscr{X} } $.

Denote $ \overline{ \mathscr{P} } _t^{ \overline{X} _0} \left(A\right) \stackrel{ \rm{ \;def} } {= \! \! \!=} \mathbf{P} \{ \overline{X} _t \in A \} $ for the process $\left( \overline{X} _t,\,t\geqslant 0\right)$ with initial state $ \overline{X} _0$ and $A \in \mathscr{B} \left(\overline{ \mathscr{X} }\right)$.

If $ \mathbf{E} \, \zeta_i< \infty$ then $ \overline{ \mathscr{P} } _t^{ \overline{X} _0} \Longrightarrow \overline{ \mathscr{P} } $.

If we can prove that $ \left \| \overline{ \mathscr{P} } _t^{ \overline{X} _0} - \overline{ \mathscr{P} } \right \|_{TV} \leqslant \varphi\left(t, \overline{X} _0\right)$ for all $t \geqslant 0$, then this inequality is true for the original non-Markov regenerative process $(X_t,\,t\geqslant 0)$.

For simplicity, we assume that the process $\left( \overline{X} _t,\,t\geqslant 0\right)$ is homogeneous Markov process, i.e. the transition function of this process in the period $[0, \theta_0]$ is the same as in the periods $[ \theta_i, \theta_{i+1} ]$, $i\geqslant1$.

Thus, in the sequel we suppose that the regenerative process $X_t$ is homogeneous Markov process.

Notice, that in general case $ \mathbf{P} \left \{ \tau \left (X_0',X_0'' \right)< \infty \right \} <1$ (for the Markov processes in continuous time), and the ``direct'' use of coupling method is impossible.

 \subsubsection{Successful coupling (see \cite{Grif}).} So, we will construct (in a special probability space) the paired stochastic process $ \mathscr{Z} _t= \left (\left (Z_t',Z_t'' \right),t \geqslant0 \right)$ such that:
\begin{enumerate}
 \item For all $t \geqslant 0 \;$ $ \;X_t' \stackrel{ \mathscr{D} } {=} Z_t'$ and $X_t'' \stackrel{ \mathscr{D} } {=} Z_t''$.
 \item $ \mathbf{E} \, \tau \left (Z_0',Z_0'' \right)< \infty$, where
$
 \tau \left (Z_0',Z_0'' \right)={ \tau} \left(\mathscr Z_0 \right) \stackrel{ \rm{def} } {= \! \! \!= \! \! \!=} \inf \left \{ t \geqslant 0: \,Z_t'=Z_t'' \right \}.
$
 \item $Z_t'=Z_t''$ for all $t \geqslant { \tau} \left (Z_0',Z_0'' \right)$.
\end{enumerate}

The paired stochastic process $ \left(\mathscr{Z} _t,\,t\geqslant 0\right)= \big (\left (Z_t',Z_t'' \right),t \geqslant0 \big)$ satisfying the conditions 1--3 is called {\it successful coupling}.
\begin{rem}
The processes $\left(Z_t',\,t\geqslant 0\right)$ and $\left(Z_t'',\,t\geqslant 0\right)$ can be non-Markov, and its finite-dimensional distributions may differ from the finite-dimensional distributions of the processes $\left(X_t',\,t\geqslant 0\right)$\linebreak and $\left(X_t'',\,t\geqslant 0\right)$ respectively.
Furthermore, the processes $\left(Z_t',\,t\geqslant 0\right)$ and $\left(Z_t',\,t\geqslant 0\right)$ may be dependent. \hfill\ensuremath{\triangleright}

\end{rem}

For all $A \in \mathscr{B} \left(\mathscr{X}\right)$ we can use the coupling inequality in the form:
 \begin{multline} \label{zvlab3}
 \left| \mathscr{P} ^{X_0'} _t\left(A\right) - \mathscr{P} _t^{X_0''} \left(A\right) \right|= \left| \mathbf{P} \left \{ X_t' \in A \right \} - \mathbf{P} \left \{ X_t'' \in A \right \} \right|=
 \\ \\
 = \left| \mathbf{P} \left \{ Z_t' \in A \right \} - \mathbf{P} \left \{ Z_t'' \in A \right \} \right| \leqslant
 \mathbf{P} \left \{ { \tau} \left (Z_0',Z_0'' \right) \geqslant t \right \}
 \leqslant \\ \\
 \leqslant \frac{\mathbf {E} \, \varphi \left ({ \tau} \left (Z_0',Z_0'' \right) \right)}{\varphi\left(t\right)} \leqslant \frac{C \left (Z_0',Z_0'' \right)}{\varphi\left(t\right)}
 \end{multline}
for some constant $C \left (Z_0',Z_0'' \right)$.
As $Z_0^{\left(i\right)} =X_0^{\left(i\right)} $, the right-hand side of the inequality depends only on $X_0^{\left(i\right)} $.
Then we can integrate the inequality (\ref{zvlab3}) with respect to the measure $ \mathscr{P} $ as in (\ref{zvlab2}):
\begin{equation}\label{eq}
 \left| \mathscr{P} ^{X_0'} _t\left(A\right) - \mathscr{P} \left(A\right) \right| \leqslant\frac{\dd \int \limits_{ \mathscr{X} } C \left (Z_0',Z_0'' \right) \mathscr{P} \left (\, \mathrm{d} Z_0'' \right)}{\varphi\left(t\right)}=\frac{ \widehat{C} \left ( Z_0' \right)}{\varphi \left(t\right)}.
\end{equation}
The inequality (\ref{eq}) implies the bounds for $ \left\| \mathscr{P} ^{X_0'} _t - \mathscr{P} \right\|_{TV}$.

{ \it However, the integration in (\ref{eq}) gives some trouble.}

 \subsection{Stationary modification of the coupling method.} We will construct a successful coupling $ \left(\mathscr Z_t,\, t\geqslant0\right)=\left(Z_t, \widetilde Z_t\right)$ for the process $\left(X_t,\, t\geqslant0\right)$ and its stationary version $ \left(\widetilde X_t,\, t\geqslant0\right)$, then we will estimate the random variable
 $$  \tau\left(X_0\right)=  \tau\left(Z_0\right) \stackrel {{ \; \rm def} } {= \! \! \!=} \inf \left \{ t>0: \, Z_t= \widetilde Z_t \right \} .$$
Then
$$ \left | \mathscr{P} ^{X_0} _t\left(A\right) - \mathscr{P} \left(A\right) \right | \leqslant \mathbf{P} \left \{  \tau\left(X_0\right)>t \right \} \leqslant \frac{ \mathbf{E} \varphi\left( \tau\left(X_0\right)\right)} { \varphi\left(t\right)}
$$
for all $A\in\mathscr{B}\left(\mathscr{X}\right)$ by Markov inequality, and
$$
\left \| \mathscr{P} ^{X_0} _t - \mathscr{P} \right \|_{TV}\bd 2 \supl_{A\in\mathscr{B}\left(\mathscr{X}\right)}\left | \mathscr{P} ^{X_0} _t\left(A\right) - \mathscr{P} \left(A\right) \right | \leqslant 2 \frac{ \mathbf{E} \varphi\left( \tau\left(X_0\right)\right)} { \varphi\left(t\right)}.
$$
Here $ \mathscr{P} ^{X_0} _t$ is the distribution of process $\left(X_t,\, t\geqslant0\right)$ with the initial state $X_0$.

Thus, the first step of our proof is the proof of the Theorem \ref{success}; the Theorem \ref{result} is the consequence of the proof of the Theorem \ref{success}.

 \section{Implementation of idea.}
 \begin{thm} \label{success}
There exists a successful coupling $ \left(\mathscr Z_t,\, t\geqslant0\right)=\left(\left(Z_t, \widetilde Z_t\right),\, t\geqslant0\right)$ for the process $\left(X_t,\, t\geqslant0\right)$ and its stationary version $\left( \widetilde X_t,\, t\geqslant0\right)$.
 \end{thm}
 \begin{proof} The proof of the Theorem \ref{success} consists of 4 steps.

~

 \noindent \textbf{1.} Let us prove the Theorem \ref{success} for the backward renewal process $\left(B_t,\, t\geqslant 0\right)$ (Notation 1).
We will give the construction of the successful coupling for the process $\left(B_t,\, t\geqslant 0\right)$ and its stationary version $ \left(\widetilde{B} _t,\, t\geqslant 0\right)$ (on the probability space $ \left(\widetilde{ \Omega} , \widetilde{ \mathscr{F} } , \widetilde{ \mathbf{P} }\right)$ in which the random variables $\UU_i$, $\UU_i'$, $\UU_i''$, $\UU_i'''$ are defined -- see Notation 5).

I.e. we will find the uniform bounds for the convergence
$$
\PP\{B_t\leqslant s\}\longrightarrow \widetilde{F}\left(s\right)
$$
as $t\longrightarrow \infty$.

\begin{center}
 \begin{figure}
 \centering
 \begin{picture}(350,40)
\put(-20,20){\vector(1,0){350}}
\thicklines
{\qbezier(-20,33)(20,34)(60,20)}
{\qbezier(60,20)(80,40)(100,20)}
\qbezier(100,20)(140,40)(180,20)
\qbezier(180,20)(200,40)(220,20)
\qbezier(220,20)(250,40)(280,20)
\qbezier(280,20)(290,40)(300,20)
\qbezier(300,20)(320,31)(330,30)
\thinlines
\put(5,35){\scriptsize$\zeta_0\sim F_a$}
\put(65,35){\scriptsize$\zeta_1\sim F$}
\put(125,35){\scriptsize$\zeta_2\sim F$}
\put(185,35){\scriptsize$\zeta_3\sim F$}
\put(235,35){\scriptsize$\zeta_4\sim F$}
\put(210,10){\scriptsize$\theta_3$}
\put(58,10){\scriptsize$\theta_0$}
\put(98,10){\scriptsize$\theta_1$}
\put(178,10){\scriptsize$\theta_2$}
\put(278,10){\scriptsize$\theta_4$}
\put(298,10){\scriptsize$\theta_5$}
\put(-20,10){\scriptsize$0$}
\put(320,10){\scriptsize${B}_t$}
\end{picture}
\begin{picture}(350,45)
\put(-20,20){\vector(1,0){350}}
\thicklines
{\qbezier(-20,33)(10,34)(40,20)}
{\qbezier(40,20)(60,40)(80,20)}
\qbezier(80,20)(120,40)(160,20)
\qbezier(160,20)(200,40)(240,20)
\qbezier(240,20)(260,40)(280,20)
\qbezier(280,20)(320,31)(330,30)
\thinlines
\put(0,35){\scriptsize$\widetilde \zeta_0\sim\widetilde F$}
\put(45,35){\scriptsize$\zeta_1\sim F$}
\put(105,35){\scriptsize$\zeta_2\sim F$}
\put(185,35){\scriptsize$\zeta_3\sim F$}
\put(255,35){\scriptsize$\zeta_4\sim F$}
\put(238,10){\scriptsize$\widetilde \theta_3$}
\put(38,10){\scriptsize$\widetilde \theta_0$}
\put(78,10){\scriptsize$\widetilde \theta_1$}
\put(158,10){\scriptsize$\widetilde \theta_2$}
\put(278,10){\scriptsize$\widetilde \theta_4$}
\put(-20,10){\scriptsize$0$}
\put(320,10){\scriptsize$\widetilde{B}_t$}
\end{picture}
 \caption{Construction of the independent processes $\left(B_t,\, t\geqslant 0\right)$ and $ \left(\widetilde{B} _t,\, t\geqslant 0\right)$}\label{inde}
 \end{figure}
\end{center}
Recall that the processes $B_t$ and $ \widetilde{B} _t$ can be constructed as follows (see Fig. \ref{inde}):

$ \zeta_0 \stackrel{ \rm{ \;def} } {= \! \! \!=} \mathbb{F} ^{-1} \left(\mathscr{U} _0\right)$, and $ \zeta_i \stackrel{ \rm{ \;def} } {= \! \! \!=} F^{-1} \left(\mathscr{U} _i\right)$ for $i>0$; $ \theta_i \stackrel{ \rm{ \;def} } {= \! \! \!=} \sum \limits_{j=0} ^i \zeta_j$; $Z_t \stackrel{ \rm{ \;def} } {= \! \! \!=} t- \max \{ \theta_i: \, \theta_i \leqslant t \} \stackrel{ \mathscr{D} } {=} B_t$.

$ \widetilde{ \theta} _0= \widetilde \zeta_0 \stackrel{ \rm{ \;def} } {= \! \! \!=} \widetilde F^{-1} \left(\mathscr{U} _1'\right)$, and $ \widetilde{ \zeta} _i \stackrel{ \rm{ \;def} } {= \! \! \!=} F^{-1} \left(\mathscr{U} _i'\right)$ for $i>0$; $ \widetilde{ \theta} _i \stackrel{ \rm{ \;def} } {= \! \! \!=} \sum \limits_{j=0} ^i \widetilde{ \zeta} _j$;
 $ \widetilde Z_0 \stackrel{ \rm{ \;def} } {= \! \! \!=} F^{-1} _{ \widetilde \theta_0} \left(\mathscr{U} _1''\right)$;

 $ \widetilde Z_t \stackrel{ \rm{ \;def} } {= \! \! \!=} \mathbf{1} \left(t< \widetilde \theta_0\right)\left(t+ \widetilde Z_0\right)+ \mathbf{1} \left(t \geqslant \widetilde \theta_0\right)\left(t- \max \{ \widetilde \theta_n: \, \widetilde \theta_n \leqslant t \}\right) \stackrel{ \mathscr{D} } {=} \widetilde B_t$.
 \begin{rem}\label{rem5}
 $\mathbf{P} \{ \widetilde Z_0 \leqslant s \} =$
\begin{multline*}
  =\int \limits_0^ \infty F_{u} \left(s\right) \, \mathrm{d} \widetilde{F} \left(u\right)= \int \limits_0^ \infty \frac{F\left(s+u\right)-F\left(u\right)} {1-F\left(u\right)} \times \frac{1-F\left(u\right)}{\mu} \, \mathrm{d} u =\frac{\dd\int \limits_0^s \left(1-F\left(u\right)\right) \, \mathrm{d} u}{\mu}= \widetilde{F} \left(s\right).
\end{multline*}
 \hfill\ensuremath{\triangleright}
 \end{rem}
This construction is the construction of independent version of the processes $\left(B_t,\, t\geqslant 0\right)$ and $ \left(\widetilde{B} _t,\, t\geqslant 0\right)$.
Now we will transform this construction.

~

 \noindent \textbf{2.}
To construct a pair of (dependent) renewal processes, it is enough to construct all renewal times of both processes (times $ \vartheta_i$ in the Fig.1).
We will construct this pair $ \left(\mathscr Z_t,\, t\geqslant 0\right)=\left(\left(Z_t, \widetilde{Z} _t\right),\, t\geqslant 0\right)$ by induction.
 \\
 \noindent { \it Basis of induction.}

 Put
$$
 \theta_0 \stackrel {{ \rm \; def} } {= \! \! \!=} \mathbb{F} ^{-1} \left(\mathscr{U} _{0}\right), \qquad \widetilde \theta_0 \stackrel {{ \rm \; def} } {= \! \! \!=} \widetilde F^{-1} \left(\mathscr{U} _{0} '\right),\qquad \widetilde Z_0 \stackrel {{ \rm \; def} } {= \! \! \!=} F_{ \widetilde \theta_0} ^{-1} \left(\mathscr{U} _{0} ''\right)
$$
(recall that $\PP\{\widetilde{Z}_0\leqslant s\}=\widetilde{F}\left(s\right)$ -- see Remark \ref{rem5}); and we put
$$Z_t \stackrel {{ \rm \; def} } {= \! \! \!=} t,\qquad \widetilde Z_t \stackrel {{ \rm \; def} } {= \! \! \!=} t+ \widetilde Z_0$$
for $t \in[0, \vartheta_0)$ where $ \vartheta_0 \stackrel {{ \rm \; def} } {= \! \! \!=} t_0 \wedge \widetilde t_0$ (in the Fig.1 $ \vartheta_0= \widetilde{ \theta} _0$).

 \noindent { \it Inductive step.}

 Suppose that we have constructed the process $ \left(\mathscr Z_t,\, t\geqslant 0\right)$ for $t \in [0, \vartheta_n)$, where $ \vartheta_n= \theta_i \wedge \widetilde \theta_j$.

There are three alternatives.

 { \it 1.} $ \vartheta_n= \theta_i= \widetilde \theta_j$ -- in the Fig. \ref{suc} this situation occurs for the first time at the point $ \vartheta_5$.

In this case we put
$$Z_{ \vartheta_n} = \widetilde Z_{ \vartheta_n} = 0, \qquad \theta_{i+1} = \widetilde \theta_{j+1} = \vartheta_{n+1} =F^{-1} \left(\mathscr{U} _{n+1}\right)+ \vartheta_{n} ;
$$
and $Z_{t} = \widetilde Z_{t} \stackrel {{ \rm \; def} } {= \! \! \!=} t- \vartheta_n$ for $t \in [ \vartheta_n, \vartheta_{n+1})$.
After the first coincidence (time $ \widetilde{ \tau} $ in the Fig. \ref{suc}) the processes $\left(Z_t,\, t\geqslant 0\right)$ and $ \left(\widetilde{Z} _t,\, t\geqslant 0\right)$ are identical.

\noindent { \it 2.} $ \vartheta_n= \widetilde \theta_j< \theta_i$ (the times $ \widetilde{ \theta} _0$ and $ \widetilde{ \theta} _3$ in the Fig. \ref{suc}).

In this case we put
$$
 \widetilde Z_{ \vartheta_n} = 0, \; \; Z_{ \vartheta_n} = Z_{ \vartheta_n-0} ,\qquad \widetilde \theta_{j+1} \stackrel {{ \rm \; def} } {= \! \! \!=} \widetilde \theta_j+F^{-1} \left(\mathscr{U} _{n+1}\right);
$$
and $ \widetilde Z_{t} \stackrel {{ \rm \; def} } {= \! \! \!=} t- \vartheta_n$, $ Z_{t} \stackrel {{ \rm \; def} } {= \! \! \!=} t- \vartheta_n+Z_{ \vartheta_n} $ for $t \in[ \vartheta_n, \vartheta_{n+1})$ where $ \vartheta_{n+1} \stackrel {{ \rm \; def} } {= \! \! \!=} \theta_i \wedge \widetilde \theta_{j+1} $.

\noindent { \it 3.} $ \vartheta_n= \theta_i< \widetilde \theta_j$ (the times $ \theta_0$, $ \theta_1$ and $ \theta_2$ in the Fig. \ref{suc}).

In this case we put
$$
 \theta_{i+1} \stackrel {{ \rm \; def} } {= \! \! \!=} \theta_i+ \Xi\left(\mathscr{U} _{n+1} , \mathscr{U} _{n+1} ', \mathscr{U} _{n+1} ''\right); \qquad \widetilde \theta_j \stackrel {{ \rm \; def} } {= \! \! \!=} \theta_i+ \widetilde A,
$$
 where $ \widetilde A = \widetilde \Xi\left(\mathscr{U} _{n+1} , \mathscr{U} _{n+1} ', \mathscr{U} _{n+1} ''\right)$;
 and
$$
Z_t \stackrel {{ \rm \; def} } {= \! \! \!=} t- \vartheta_n, \qquad \widetilde Z_t \stackrel {{ \rm \; def} } {= \! \! \!=} t- \vartheta_n + F^{-1} _{ \widetilde A} \left(\mathscr{U} _{n+1} '''\right)
$$
for $t \in[ \vartheta_n, \vartheta_{n+1})$, where $ \vartheta_{n+1} \stackrel {{ \rm \; def} } {= \! \! \!=} \theta_i \wedge \widetilde \theta_{j+1} $.

~

\noindent \textbf{3.} Let us prove that \emph{the process $ \left(\mathscr Z_t,\,t\geqslant0\right)=\left(\left(Z_t, \widetilde Z_t\right),\,t\geqslant0\right)$ is a successful coupling for the processes } $\left(X_t,\,t\geqslant0\right)$ and $\left(\widetilde{X}_t,\,t\geqslant0\right)$, and
$$ \mathbf{E} \, { \tau} \left(X_0\right) \leqslant\mathbf{E} \, \widetilde{ \tau} \left(X_0\right) \leqslant \mathbf {E} \, \zeta_0 + 2 \varkappa^{-1} \mathbf {E} \, \zeta_1< \infty.
$$

Denote
$$ \mathscr{E} _n \stackrel{ \rm{ \;def} } {= \! \! \!=} \{ Z_{ \theta_{n} } = \widetilde{Z} _{ \theta_{n} } |Z_{ \theta_{n-1} } \neq \widetilde{Z} _{ \theta_{n-1} } \} ; \qquad \mathbf{P} \left(\mathscr{E} _n\right)= \varkappa.
$$
According to our construction of the pair $ \left(\mathscr{Z} _t\,t\geqslant0\right)$, we have $ \mathbf{P} \{Z_{ \theta_0} \neq \widetilde{Z} _{ \theta_0} \} =1$ because the distribution $ \widetilde{F} \left(s\right)$ is absolutely continuous.

Then
$$ \mathbf{P} \{Z_{ \theta_1} = \widetilde{Z} _{ \theta_1} \} = \mathbf{P} \left(\mathscr{E} _1\right)= \varkappa, \qquad \mathbf{P} \{Z_{ \theta_{n+1} } = \widetilde{Z} _{ \theta_{n+1} } \; \& \;Z_{ \theta_n} \neq \widetilde{Z} _{ \theta_n} \} = \mathbf{P} \left (\mathscr{E} _n \prod \limits_{i=0} ^{n-1} \overline{ \mathscr{E} } _i \right),
$$
and
$$ \mathbf{P} \{ \widetilde{ \tau} = \theta_{n+1} \} = \mathbf{P} \left(\mathscr{E} _n\right) \prod \limits_{i=0} ^{n-1} \mathbf{P} \left(\overline{ \mathscr{E} } _i\right)= \varkappa \left(1- \varkappa\right)^n.
$$

Now, using the inequality
 \begin{equation} \label{zvlab4} 
 \mathbf{E} \left(\xi \times \mathbf{1} \left(\mathscr{E}\right)\right) \mathbf{P} \left(\mathscr{E}\right) \leqslant \mathbf{E} \, \xi
 \end{equation}
for non-negative random variable $ \xi$, and considering that
$$ \mathbf{E} \left (\zeta_n \mathbf{1} \left (\bigcap \limits_{i=1} ^{n-1} \overline{ \mathscr{E} _i} \right) \right) = \mathbf{E} \, \zeta_n
$$ for $n>0$, we have
\begin{multline}\label{zvlab5}
 \mathbf{E} \widetilde{ \tau} = \mathbf{E} \, \zeta_0+ \mathbf{E} \left(\zeta_1 \times \mathbf{1} \left(\mathscr{E} _1\right)\right) \mathbf{P} \left(\mathscr{E} _1\right)+ \mathbf{E} \left(\left(\zeta_1+ \zeta_2\right) \times \mathbf{1} \left(\overline {\mathscr{E}} _1\right)\right) \mathbf{1} \left(\mathscr{E} _2\right) \mathbf{P} \left(\overline {\mathscr{E}} _1 \mathscr{E} _2 \right)+
 \\ \\
 +\mathbf{E} \left(\left(\zeta_1+ \zeta_2+ \zeta_3\right) \times \mathbf{1} \left(\overline {\mathscr{E}} _1\right) \mathbf{1} \left(\overline {\mathscr{E}} _2\right) \mathbf{1} \left(\mathscr{E} _3\right)\right) \mathbf{P} \left(\overline {\mathscr{E}} _1 \overline {\mathscr{E}} _2 \mathscr{E} _3\right)+
 \ldots +
 \\ \\
 + \mathbf{E} \left (\left (\sum \limits_{i=1} ^n \zeta_i \right) \times \mathbf{1} \left(\mathscr{E} _n\right) \prod \limits_{i=1} ^{n-1} \mathbf{1} \left(\overline {\mathscr{E}} _i\right) \right) \mathbf{P} \left (\mathscr{E} _n \prod \limits_{i=1} ^{n-1} \overline {\mathscr{E}} _i \right)+ \ldots \leqslant
 \\ \\
 \leqslant \mathbf{E} \, \zeta_0+ \mathbf{E} \, \zeta_1 \left (1+ \varkappa \sum \limits_{i=0} ^ \infty\left(1- \varkappa\right)^i \right) +\left(1- \varkappa\right) \mathbf{E} \, \zeta_2 \left (1+ \varkappa \sum \limits_{i=0} ^ \infty \left(1- \varkappa\right)^i \right) +
 \\ \\
 + \ldots + \left(1- \varkappa\right)^{n-1} \mathbf{E} \, \zeta_n \left (1+ \varkappa \sum \limits_{i=0} ^ \infty \left(1- \varkappa\right)^i \right) + \ldots= \mathbf{E} \, \zeta_0 + 2 \varkappa^{-1} \mathbf{E} \, \zeta_1.
\end{multline}

\noindent Theorem \ref{success} is proved.
 \end{proof}
\begin{center}
 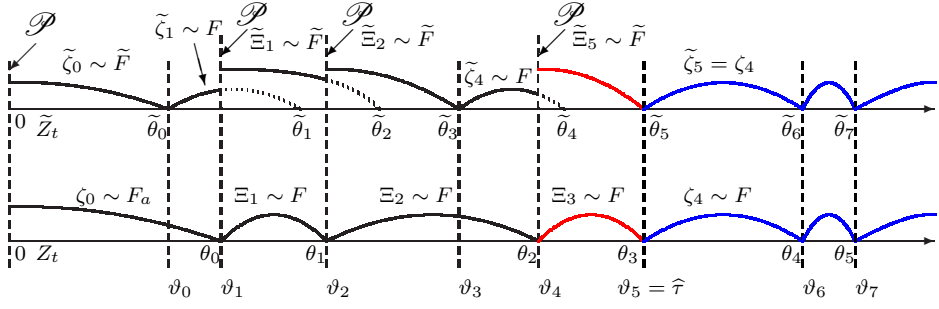
\begin{figure}
 \centering
 \begin{picture}(350,90)
\thinlines
\put(-20,70){\vector(1,0){350}}
\put(-20,20){\vector(1,0){350}}
\put(-10,95){\vector(-1,-1){10}}
\thicklines
\multiput(40,10)(0,5){16}{\line(0,1){3}}
\multiput(150,10)(0,5){16}{\line(0,1){3}}
\multiput(-20,10)(0,5){16}{\line(0,1){3}}
{\qbezier(-20,80)(10,81)(40,70)}
{\qbezier(-20,33)(20,34)(60,20)}
\qbezier(40,70)(65,85)(90,70)
\multiput(60,10)(0,5){18}{\line(0,1){3}}
\multiput(61,71)(2,0){20}{\color{white}\line(0,1){20}}
{\qbezier(60,20)(80,40)(100,20)}
\multiput(100,10)(0,5){18}{\line(0,1){3}}
\qbezier(60,85)(105,86)(120,70)
\multiput(101,71)(2,0){20}{\color{white}\line(0,1){20}}
\qbezier(100,20)(140,40)(180,20)
\multiput(180,10)(0,5){18}{\line(0,1){3}}
\qbezier(100,85)(125,86)(150,70)
\qbezier(150,70)(170,85)(190,70)
\multiput(181,71)(2,0){16}{\color{white}\line(0,1){20}}
\rr\qbezier(180,20)(200,40)(220,20)\bk
\rr\qbezier(180,85)(200,86)(220,70)\bk
\multiput(220,10)(0,5){16}{\line(0,1){3}}
\bb\qbezier(220,20)(250,40)(280,20)\bk
\multiput(280,10)(0,5){16}{\line(0,1){3}}
\bb\qbezier(220,70)(250,90)(280,70)
\qbezier(280,20)(290,40)(300,20)
\qbezier(280,70)(290,90)(300,70)
\qbezier(300,20)(320,31)(330,30)\bk
\multiput(300,10)(0,5){16}{\line(0,1){3}}
\bb\qbezier(300,70)(320,81)(330,80)\bk
\thinlines
\put(-15,98){$\PPP$}
\put(0,85){\scriptsize${\widetilde\zeta_0\sim \widetilde F}$}
\put(50,95){\vector(1,-4){4}}
\put(35,98){\scriptsize$\widetilde\zeta_1\sim F$}
\put(70,100){\vector(-1,-1){10}}
\put(65,102){$\PPP$}
\put(71,93){\scriptsize${\widetilde\Xi_1\sim \widetilde F}$}
\put(110,100){\vector(-1,-1){10}}
\put(105,102){$\PPP$}
\put(113,94){\scriptsize${\widetilde\Xi_2\sim \widetilde F}$}
\put(152,80){\scriptsize$\widetilde\zeta_4\sim F$}
\put(190,100){\vector(-1,-1){10}}
\put(185,102){$\PPP$}
\put(193,94){\scriptsize${\widetilde\Xi_5\sim \widetilde F}$}
\put(5,35){\scriptsize$\zeta_0\sim F_a$}
\put(65,35){\scriptsize$\Xi_1\sim F$}
\put(120,35){\scriptsize$\Xi_2\sim F$}
\put(185,35){\scriptsize$\Xi_3\sim F$}
\put(235,35){\scriptsize$\zeta_4\sim F$}
\put(235,85){\scriptsize$\widetilde \zeta_5= \zeta_4$}
\put(210,13){\scriptsize$\theta_3$}
\put(52,13){\scriptsize$\theta_0$}
\put(92,13){\scriptsize$\theta_1$}
\put(172,13){\scriptsize$\theta_2$}
\put(272,13){\scriptsize$\theta_4$}
\put(292,13){\scriptsize$\theta_5$}
\put(-10,13){\scriptsize$Z_t$}
\put(-10,60){\scriptsize$\widetilde Z_t$}
\put(-18,13){\scriptsize$0$}
\put(-18,63){\scriptsize$0$}
\put(32,60){\scriptsize$\widetilde \theta_0$}
\put(87,60){\scriptsize$\widetilde \theta_1$}
\put(117,60){\scriptsize$\widetilde \theta_2$}
\put(142,60){\scriptsize$\widetilde \theta_3$}
\put(187,60){\scriptsize$\widetilde \theta_4$}
\put(222,60){\scriptsize$\widetilde \theta_5$}
\put(272,60){\scriptsize$\widetilde \theta_6$}
\put(292,60){\scriptsize$\widetilde \theta_7$}
\put(40,0){\scriptsize$\vartheta_0$}
\put(60,0){\scriptsize$\vartheta_1$}
\put(100,0){\scriptsize$\vartheta_2$}
\put(150,0){\scriptsize$\vartheta_3$}
\put(180,0){\scriptsize$\vartheta_4$}
\put(210,0){\scriptsize$\vartheta_5=\widehat\tau$}
\put(280,0){\scriptsize$\vartheta_6$}
\put(300,0){\scriptsize$\vartheta_7$}
\end{picture}
 \caption{Construction of the successful coupling $(\mathscr{Z}_t,\,t\geqslant0).$ }\label{suc}
 \end{figure}
\end{center}

~

 \noindent \textbf{4.} Now we return to the Markov regenerative process $\left(X_t,\,t\geqslant 0\right)$.
The regeneration times $ \theta_0$, $ \theta_1$, \ldots of the process $\left(X_t,\,t\geqslant 0\right)$ form an embedded backward renewal process $\left(Y_t,\,t\geqslant 0\right)$; $Y_t \stackrel {{ \rm \; def} } {= \! \! \!=} t- \max \left \{ \theta_i: \, \theta_i \leqslant t \right \} $ (the backward renewal time of the embedded renewal process) with the distribution of the renewal time $ \mathbf{P} \left \{ \zeta \leqslant s \right \} =F\left(s\right)$; the length of $i$-th regeneration period is $ \theta_i- \theta_{i-1} = \zeta_i \stackrel{ \mathscr{D} } {=} \zeta \; \; $ ($i>1$).

{ \it We will apply the coupling method for extended process $ \left(\mathbb{X} _t,\,t\geqslant 0\right)=\left(\left(X_t,Y_t\right),\,t\geqslant 0\right)$ on the extended state space $ \widehat{ \mathscr{X} } = \mathscr{X} \times \mathbb{R} _+$.}

The random element $ \mathfrak{X} _i= \left \{ X_t, t \in[ \theta_{i-1} , \theta_i) \right \} $ depends on the random variable $ \zeta_i= \theta_i- \theta_{i-1} $; for $A \in \mathscr{B} \left(\mathscr{X}\right)$ denote

$$
 \mathscr {G} _a\left(s,A\right) \stackrel {{ \rm \; def} } {= \! \! \!=} \mathbf{P} \left \{ X_{ \theta_{i-1} +s} \in A| \zeta_i=a \right \} , \; \; s \in[0,a);
$$
$ \mathscr {G} _a\left(s,A\right)$ specify a conditional distribution of $\left(X_t,\,t\geqslant 0\right)$ on the time period $[ \theta_{i-1} , \theta_i)$ given $ \left \{ \theta_i- \theta_{i-1} =a \right \} $.

Therefore if we know all regeneration times of the process $\left(X_t, \,t \geqslant 0\right)$, then we know the (conditional) distribution of the process $\left(X_t,\,t\geqslant 0\right)$ in every time after the first regeneration time $t_0$: this distribution is determined by the conditional distribution $ \mathscr G_{ \zeta_i} \left(s,A\right)$ of the random elements $ \mathfrak X_i$.

Also the first regeneration time depends on the initial state; denote
$$
 \mathscr {H} _a\left(A\right) \stackrel {{ \rm \; def} } {= \! \! \!=} \mathbf{P} \left \{ X_0 \in A| \theta_0=a \right \} = \mathbf{P} \{X_t \in A|\left(\min \{ \theta_i: \, \theta_i \geqslant t \} -t\right)=a \} ;
$$
$ \mathscr {H} _a\left(A\right)$ specify a conditional distribution of $\left(X_t,\,t\geqslant 0\right)$ given $ \{ \zeta_t^{ \ast} =a \} $, where \linebreak$ \zeta_t^{ \ast} \stackrel{ \rm{ \;def} } {= \! \! \!=} \min \{t_i: \,t_i \geqslant t \} -t$ is a residual time of regeneration period at the time $t$.


Now we will construct the successful coupling for extended process $ \left(\mathbb{X} _t,\,t\geqslant 0\right)=\left(\left(X_t,Y_t\right),\,t\geqslant 0\right)$ and its stationary version $ \left(\widetilde{ \mathbb{X} } _t,\,t\geqslant 0\right)=\left(\left(\widetilde X_t, \widetilde Y_t\right),\,t\geqslant 0\right)$.

For this aim we construct the successful coupling $ \left(\mathscr{W},\,t\geqslant 0\right) _t=\left(\left(W_t, \widetilde{W} _t\right),\,t\geqslant 0\right)$ for the backward renewal process $\left(Y_t,\,t\geqslant 0\right)$ and its stationary version $\left( \widetilde{Y} _t,\,t\geqslant 0\right)$ considering that the first renewal time $ \theta_0$ of the process $\left(Y_t,\,t\geqslant 0\right)$ has a distribution $ \mathbb{F} \left(s\right)$.

After construction of renewal points $ \{ \theta_i \} $ of the process $\left(W_t,\,t\geqslant 0\right)$ and renewal points $ \{ \widetilde \theta_i \} $ of the process $\left( \widetilde W_t,\,t\geqslant 0\right)$ we can complete them to the pairs $\left(\left(Z_t,W_t\right),\,t\geqslant 0\right) \stackrel{ \mathscr{D} } {=} \left(\left(X_t,Y_t\right),\,t\geqslant 0\right)$ and $ \left(\left(\widetilde Z_t, \widetilde W_t \right),\,t\geqslant 0\right) \stackrel{ \mathscr{D} } {=} \left(\left(\widetilde X_t, \widetilde Y_t \right),\,t\geqslant 0\right)$ by using $ \mathscr {G} _a\left(s,A\right)$ and $ \mathscr {H} _a\left(A\right)$.

In the construction of the processes $\left(W_t,\,t\geqslant 0\right)$ and $\left( \widetilde W_t,\,t\geqslant 0\right)$ we can apply the technics used in the proof of the Theorem \ref{success} in such a way that
$$
 \tau\left(X_0\right)= \inf \left \{t: \,W_t= \widetilde W_t \right \} \leqslant t_0+ \sum \limits_{i=1} ^ \nu \zeta_i,
$$
where $ \mathbf{P} \{ \nu>n \} =\left(1- \varkappa\right)^n$.

So, for $t \geqslant  \tau\left(X_0\right)$ we have $W_{t} = \widetilde W_{t} $, by construction of the backward renewal processes $\left(W_t,\,t\geqslant 0\right)$ and $ \left(\widetilde W_t,\,t\geqslant 0\right)$.

Then for $t \geqslant  \tau\left(X_0\right)$ we have
$$
\mathscr{P} ^{X_0} _t\left(A\right)= \mathbf{P} \{X_t \in A \} = \mathbf{P} \{ \widetilde X_t \in A \} = \mathscr{P} \left(A\right)
$$
as the distribution of $X_t$ and $ \widetilde X_t$ (after the first renewal point) is determined only by renewal points.

 \begin{proof} [Proof of the Theorem \ref{result}.]~

 \textbf{1.} Using the inequality (\ref{zvlab4}) and Jensen's inequality for $k \geqslant 1$ and $a_i\leqslant 0$ in the form
 $$
 \left (\sum \limits_{i=1} ^n a_i \right)^k \leqslant n^{k-1} \sum \limits_{i=1} ^n a_i^k
$$
we have from (\ref{zvlab5}):
\begin{multline*}
\mathbf{E} \left({ \tau} \left(X_0\right)\right)^k \leqslant\mathbf{E} \left(\widetilde{ \tau} \left(X_0\right)\right)^k \leqslant
\\ \\
\leqslant\mathbf{E} \left (\mathbf{P} \left (\mathscr{E} _n \prod \limits_{i=1} ^{n-1} \overline{ \mathscr{E} } _i \right) \sum \limits_{n=1} ^ \infty \left (\left(n+1\right)^{k-1} \left (\zeta_0^k+ \sum \limits_{i=1} ^n \zeta_i^k \right) \mathbf{1} \left(\mathscr{E} _n \right) \prod \limits_{i=1} ^{n-1} \mathbf{1} \left(\overline{ \mathscr{E} } _i\right) \right) \right) \leqslant
 \\ \\
\leqslant \mathbf{E} \, \zeta_0^k \varkappa \sum \limits_{n=1} ^ \infty \left (\left(n+1\right)^{k-1} \left(1- \varkappa\right)^{n-1} \right)+\mathbf{E} \, \zeta_1^k \sum \limits_{n=1} ^ \infty \left (\left (\varkappa n\left(n+2\right)^{k-1} + \left(n+1\right)^{k-1} \right) \left(1- \varkappa\right)^{n-1} \right)=
\\ \\
=C\left(X_0,k\right)\Big[=\mathscr{C}(\zeta_0,k)\Big],
\end{multline*}
this inequality implies the statement 1 of the Theorem \ref{result}.

 \textbf{2.}
Using the inequality (\ref{zvlab4}) and considering that $\EE\,e^{\dd a \zeta_1\1\left(\overline{\EEE}_i\right)}=\mathfrak{P}_a\left(\zeta_1\right)$ we have from (\ref{zvlab5}):
\begin{multline*}
\mathbf{E} \,e^{\dd \alpha \tau \left(X_0\right)} \leqslant \mathbf{E} \,e^{\dd \alpha \widetilde{\tau} \left(X_0\right)} \leqslant \mathbf{E} \left (\mathbf{P} \left (\mathscr{E} _n \prod \limits_{i=1} ^{n-1} \overline{ \mathscr{E} } _i \right) \times \sum \limits_{n=1} ^ \infty \left (e^{\dd \alpha \left (\zeta_0+ \sum \limits_{i=1} ^n \zeta_i \right)} \mathbf{1} \left(\mathscr{E} _n \right) \prod \limits_{i=1} ^{n-1} \mathbf{1} \left(\overline{ \mathscr{E} } _i\right) \right) \right) \leqslant
\\ \\
\leqslant
\mathbf{E} \,e^{\dd \alpha \zeta_0} \sum \limits_{n=1} ^ \infty\left (\left(1-\varkappa\right)^n \prod\limits_{i=1}^n\mathbf{E} \left (e^{\dd \alpha \zeta_1}\1\left(\overline{\mathscr{E}}_i\right) \right)\right)=
\mathbf{E} \,e^{\dd \alpha \zeta_0}\sum \limits_{n=1} ^ \infty \left ( \left(1-\varkappa\right)^n \left (\intl_0^\infty e^{\dd \alpha s}\ud \frac{\Psi\left(s\right)}{1-\varkappa}\right)\right)=
\\ \\
=
\mathbf{E} \,e^{\dd \alpha \zeta_0}\sum \limits_{n=1} ^ \infty\left (\mathfrak{P}_\alpha\left(\zeta_1\right)\right)^n
=\frac{\mathfrak{P}_\alpha\left(\zeta_1\right) \mathbf{E} \,e^{\dd \alpha \zeta_0} } {1- \mathfrak{P}_\alpha\left(\zeta_1\right) } = \mathfrak{C} \left(X_0,\alpha\right)\Big[=\mathcal{C}\left(\zeta_0,\alpha\right)\Big],
\end{multline*}
this inequality implies the statement 2 of the Theorem \ref{result}.
 \end{proof}
 \section{Applying to the queueing theory.} In the queuing theory the distribution of the regenerative process period is often unknown.
But often the regeneration period can be split into two parts, usually this is a busy period and an idle period.
And as a rule the idle period has a known non-discrete distribution.
So, in this situation the queueing process has an embedded \emph{alternating renewal process}.

If the bounds of moments of a busy period are also known, then we can apply our construction for embedded alternating renewal process by some modification.

This modification is a construction of a successful coupling for alternating renewal process $\left(X_t,\, t\geqslant 0\right)$ and its stationary version, namely.

Let $\left(X_t,\, t\geqslant 0\right)$ be an alternating renewal process having two states, 1 and 2, say.
The time of the stay of the process in a state $i$ has the distribution function $F_i\left(s\right)= \mathbf{P} \left \{ \zeta^{\left(i\right)} \leqslant s \right \} $, and periods of stay of the process $\left(X_t,\, t\geqslant 0\right)$ in states 1 and 2 alternate.
This process is non-Markov.

We complement the state of the process by the time, during which the process located continually in this state: if the completed process is the process $\left(Y_t,\, t\geqslant 0\right)=\left(\left(n_t,x_t\right),\, t\geqslant 0\right)$ (denote $n\left(Y_t\right)=n_t$, $x\left(Y_t\right)=x_t$), then at the time $t$ the process is in the state $n_t$, and \linebreak$x_t \stackrel{ \rm{ \;def} } {= \! \! \!=} t-\sup \{s<t: \,n\left(Y_s\right) \neq n_t \} $ (for definiteness, we assume $\mathbb{F}\left(s\right)=F_a\left(s\right)$, $x_0=a$, and $x_t=a+t$ for $t\in[0, \inf\{s>0:\,n_s\neq n_0\})$).

The Markov regenerative process $\left(Y_t,\, t\geqslant 0\right)$ has a state space $ \left \{ 1,2 \right \} \times \mathbb R_+$.

Denote $c\left(n\right) \stackrel {{ \rm \; def} } {= \! \! \!=} \mathbf{1} \left(n=2\right)+2 \times \mathbf{1} \left(n=1\right)$: here $c\left(1\right)=2$, $c\left(2\right)=1$.

If $Y_0=\left(i,a\right)$, then the process $Y_t$ changes its first component $n\left(Y_t\right)$ at the times \linebreak $ \zeta^{\left(a,i\right)} = \theta_{0,i} $, $ \theta_{0,c\left(i\right)} $, $ \theta_{1,i} $, $ \theta_{1,c\left(i\right)} $, \ldots.

Denote
$$ \zeta^{\left(i\right)} _j= \theta_{j,c\left(i\right)} - \theta_{j,i} \stackrel{ \mathscr{D} } {=} \zeta^{\left(i\right)} ; \qquad \zeta^{\left(c\left(i\right)\right)} _j= \theta_{j,i} - \theta_{j-1,c\left(i\right)} \stackrel{ \mathscr{D} } {=} \zeta^{\left(c\left(i\right)\right)} ;
$$
$$\mathbf{P} \left \{ \zeta^{\left(a,i\right)} \leqslant s \right \} = \frac{F_i\left(a+s\right)-F_i\left(s\right)} {1-F_i\left(s\right)} .
$$

We assume that the distribution function $F_1\left(s\right)$ satisfies the condition ($ \ast$), and random variables $ \zeta^{\left(a,i\right)} $, $ \zeta^{\left(1\right)} _j$, $ \zeta^{\left(2\right)} _j$ are mutually independent.

Suppose that $ \mathbf{E} \varphi \left (\zeta^{\left(i\right)} \right)< \infty$ for some increasing positive function $ \varphi\left(t\right)$.

If $ \mathbf{E} \zeta^{\left(i\right)} = \mu_i< \infty$, then distribution $ \mathscr{P} _t^{Y_0} $ of the process $Y_t$ with every initial state $Y_0$ weakly convergent to the stationary distribution $ \mathscr{P} $; for the stationary version $ \widetilde Y_t$ of the process $Y_t$ we have
$$ \mathbf{P} \{n\left(\widetilde Y_t\right)=1 \} = \frac{ \mu_1} { \mu_1+ \mu_2} \stackrel{ \rm{ \;def} } {= \! \! \!=} p.
$$
If we know only an estimate $ \mu_2 \leqslant m_2$, then
$$
p \geqslant \rho \stackrel{ \rm{ \;def} } {= \! \! \!=} \frac{ \mu_1} { \mu_1+m_2} .
$$

For construction of successful coupling $\left(\mathscr{Z}_t,\, t\geqslant 0\right)=\left(\left(Z_t, \widetilde Z_t\right),\, t\geqslant 0\right)$ for the processes \linebreak$\left(Y_t,\, t\geqslant 0\right)$ and $ \left(\widetilde Y_t,\, t\geqslant 0\right)$ we will again construct the times when at least one of them change the first component.

At the times $t'_i$ such that $n\left(Y_{t_i'-0}\right)=2$ and $n\left(Y_{t_i'+0}\right)=1$ we use (with probability $p$) the random variables $ \Xi_i$ for $Y_t$ and $ \widetilde \Xi_i$ for $ \widetilde Y_t$;
$$
 \mathbf{P} \{ \Xi_i \leqslant s \} =F_1\left(s\right),\qquad \mathbf{P} \{ \widetilde \Xi_i \leqslant s \} = \widetilde F_1\left(s\right)=\frac{ \dd\int \limits_0^s \left (1- F_1\left(u\right) \right) \, \mathrm{d} u}{\mu_1},
$$
and
$$
\mathbf{P} \left\{ \Xi_i= \widetilde \Xi_i \right\} = \varkappa= \int \limits_{ \left \{s: \, \exists\; F_1'\left(s\right) \right \} } F_1'\left(s\right) \wedge \widetilde F_1'\left(s\right) \, \mathrm{d} s.
$$

And with probability $q=1-p$ at the time $ \theta_i'$ we use a procedure of the prolongation of alternating renewal process $ \widetilde Y_t$ by using the distribution $ \widetilde F_2\left(s\right)=\left(\mu_2\right)^{-1} \int \limits_0^s\left(1-F_2\left(u\right)\right) \, \mathrm{d} u$.

So,
$$
\widetilde \tau\left(Y_0\right) \stackrel{ \rm{ \;def} } {= \! \! \!=} \inf \left \{t>0: \,Z_t= \widetilde Z_t \right \} \leqslant \theta_1'+ \zeta_ \nu^{\left(1\right)} + \sum \limits_{i=1} ^{ \nu-1} \left (\zeta_i^{\left(1\right)} + \zeta_i^{\left(2\right)} \right),
$$
where $ \mathbf{P} \{ \nu>n \} =\left(1-p \varkappa\right)^n \Big[\leqslant \left(1- \rho \varkappa\right)^n \Big]$.

Hence, if $ \mathfrak{P}_a\left(\zeta_1\right)\mathbf{E}\, e^{\dd a \zeta^{\left(1\right)} }<1$, we can find a bound $ \mathfrak{C} \left(Y_0,a\right)$ for $ \mathbf{E}\, e^{\dd a \widetilde{ \tau} \left(Y_0\right) }$:
$$
\mathbf{E}\,e^{\dd a \widetilde{ \tau} \left(Y_0\right) } \leqslant \mathfrak{C} \left(Y_0,a\right).
$$

Therefore we have
$$
\left \| \mathscr{P} _t^{Y_0} - \mathscr{P} \right \|_{TV} \leqslant 2 \frac{\mathfrak{C} \left(Y_0,a\right)}{e^{\dd a t}}.
$$

Furthermore, if the distributions $F_1$ and $F_2$ satisfy the condition $\left(\ast\right)$, then the inequality $\EE\,{a\zeta^{\left(i\right)}}<\infty$ for some $a>0$ implies an existence of $\alpha>0$ such that $\mathfrak{P}_\alpha\left (\zeta^{\left(i\right)}\right)<1$ , and we can find the bounds $ \mathfrak{C} \left(Y_0,a\right)$ for $ \mathbf{E}\,{a \widetilde{ \tau} \left(Y_0\right)}$.

And if $ \mathbf{E} \left (\zeta^{\left(i\right)} \right)^K< \infty$ for some $K \geqslant1$, then we can estimate $ \mathbf{E} \left(\widetilde{ \tau} \left(Y_0\right) \right)^k$ for $k \in[1,K]$:
$$
\mathbf {E} \left(\tau\left(Y_0\right) \right)^k \leqslant C\left(Y_0,k\right),
$$
where the constant $ C\left(Y_0,k\right)$ is calculated by the schema similar to the schema of the proof of the Theorem \ref{result}.

Again we have
$$
\left \| \mathscr{P} _t^{Y_0} - \mathscr{P} \right \|_{TV} \leqslant 2\frac{C\left(Y_0,k\right)}{t^{k}}.
$$

If the regeneration period of queueing process $\left(Q_t,\,t\geqslant 0\right)$ can be split into two independent parts, then this process has an embedded alternating renewal process.

Firstly we will extend this queueing process $\left(Q_t,\,t\geqslant 0\right)$ to the Markov process $\left(X_t,\,t\geqslant 0\right)$.

Then we will complete the process $\left(X_t,\,t\geqslant 0\right)$ by the embedded alternating renewal process $\left(Y_t,\,t\geqslant 0\right)$ competed by the time from the last change of its state (as in previous part).

Using the technique of the proof of the Theorem \ref{result} we can find the bounds for the convergence rate for extended queueing process; also this bounds are useful for the original process $\left(Q_t,\,t\geqslant 0\right)$.

\textbf{Acknowledgements. } The author is grateful to A.Yu.Veretennikov for constant attention to this work, and to V.V.Kozlov for useful discussion, and to Yu.S.~Semenov for invaluable help.

\end{document}